\documentclass[12pt]{amsart}
\usepackage{amsthm}
\usepackage{amsmath}
\usepackage{amssymb}
\usepackage{mathrsfs}
\usepackage{enumerate}
\usepackage{float}
\usepackage{pgf}
\usepackage{url, multicol}
\usepackage[all]{xy}
\usepackage{graphicx}
\usepackage{hyperref}

\newtheorem{thm}{Theorem}

\theoremstyle{definition}
\newtheorem{rem}[thm]{Remark}

\numberwithin{table}{thm}

\title[Small gaps between almost primes]{Small gaps between almost primes, the parity problem, and some conjectures of
Erd\H{o}s on consecutive integers II}
\author[Goldston, Graham, Panidapu, Pintz, Schettler, and Y{\i}ld{\i}r{\i}m]{Daniel A. Goldston, Sidney W. Graham, Apoorva Panidapu, Janos Pintz$^{*}$, Jordan Schettler, and Cem Y. Y{\i}ld{\i}r{\i}m}
\date{}
\thanks{$^{*}$ Research supported by the National Research Development and Innovation Office,
NKFIH, K 119528}

\begin{document}

\maketitle

\section{Introduction}

This paper is intended as a sequel to \cite{Gold2} written by four of the coauthors here. In the paper, they proved a stronger form of the Erd\H{o}s-Mirksy 
conjecture which states that there are infinitely many positive integers $x$ such that $d(x)=d(x+1)$ where $d(x)$ denotes the number of divisors of $x$. This 
conjecture was first proven by Heath-Brown in 1984 \cite{Heat}, but the method did not reveal the nature of the set of values $d(x)$ for such $x$. In particular, one 
could not conclude that there was any particular value $A$ for which $d(x)=d(x+1)=A$ infinitely often. In \cite{Gold2}, the authors showed that there are infinitely 
many positive integers $x$ such that both $x$ and $x+1$ have exponent pattern $\{2,1,1,1\}$, so
\begin{equation}\label{24}
d(x)=d(x+1)=24 \mbox{ for infinitely many positive integers $x$.}
\end{equation}
Similar results were known for certain shifts $n$, i.e., $x$ and $x+n$ have the same exponent pattern infinitely often. This was 
done for shifts $n$ which are either even or not divisible by the product of a pair of twin primes. The goal of this paper is to give simple proofs of results on 
exponent patterns for an arbitrary shift $n$.

\section{Notation and Preliminaries}

For our purposes, a \emph{linear form} is an expression $L(m) = am + b$ where
$a$ and $b$ are integers and $a > 0$. We view $L$ both as a polynomial and as a function in $m$. We say $L$ is \emph{reduced} if $\gcd(a,b) = 1$. 
If $K(m) = cm+d$ is another linear form, then a \emph{relation} between $L$
and $K$ is an equation of the form $|c_L \cdot L - c_K \cdot K| = n$ where
$c_L$,$c_K$, $n$ are all positive integers. We call $c_L$, $c_K$ the \emph{relation coefficients} and we call $n$ the \emph{relation value}.
We define the \emph{determinant} of $L$ and $K$ as $\det(L,K)=|ad-bc|$.

For
a prime $p$, a $k$-tuple
of linear forms $L_1, L_2, \ldots, L_k$ is called \emph{$p$-admissible} if
there is an integer $t_p$ such that
\begin{equation*}
L_1(t_p)L_2(t_p)\cdots L_k(t_p) \not\equiv 0 \pmod{p}
\end{equation*}
We say that a $k$-tuple of linear forms is \emph{admissible} if it is
$p$-admissible for every prime $p$. Note that a $k$-tuple of linear forms
is admissible iff all the forms are reduced and the tuple is $p$-admissible for every prime $p\leq k$.

An \emph{$E_r$ number} is a positive integer that is the product of
$r$ distinct primes. Several of the coauthors here proved the following result on $E_2$-numbers in admissible triples in \cite{Gold1}. Later, Frank Thorne \cite{Thor}
obtained a generalization for $E_r$-numbers with $r\geq 3$.

\begin{thm}\label{basic}
Let $C$ be any constant. If $L_1$, $L_2$, $L_3$ is an admissible triple of linear forms, 
then there are two among them, say $L_j$ and $L_k$ such that
both $L_{j}(x)$ and $L_k(x)$ are $E_2$-numbers with both prime factors larger than $C$ for infinitely many $x$.
\end{thm}

The results obtained in this paper will use Theorem \ref{basic} above in combination with Theorem \ref{adjoin} below, a special
case of which was proven in the previous paper \cite{Gold2}. We provide a proof here
of the general version since it contains important ideas relevant for the rest of the paper.

\begin{thm}[Adjoining Primes]\label{adjoin}
Assume that $L_i = a_i m+b_i$ for $i=1, \ldots, k$ gives an 
admissible $k$-tuple with
relations $|c_{i,j}L_i-c_{j,i}L_j| = n_{i,j}$. We can always 
``adjoin'' prime factors to the relation coefficients
without changing the relation values: for every choice of positive integers $r_1$, $r_2$, $\ldots$, $r_k$ such that $\gcd(r_i,a_i) = \gcd(r_i, \det(L_i,L_j))=\gcd(r_i,r_j)=1$ whenever $i\neq j$, there is an admissible 
$k$-tuple of linear forms $K_1,K_2, \ldots, K_k$ with relations
$|c_{i,j}r_iK_i-c_{j,i}r_jK_j| = n_{i,j}$.
\end{thm}
\begin{proof}
Let $x$ be a solution
of the congruences $L_i(x)\equiv r_i \pmod{r_i^2}$ for $1\leq i \leq k$.
Such an $x$ exists by the Chinese Remainder Theorem since $\gcd(a_i,r_i) = \gcd(r_i,r_j)=1$. This $x$ is unique
modulo $r=(r_1r_2\cdots r_k)^2$.
Now define a new $k$-tuple via $K_i(m) = L_i(rm+x)/r_i$. By construction, we have
$|c_{i,j}r_iK_i-c_{j,i}r_jK_j| = n_{i,j}$, so we only need
to check that this new $k$-tuple is admissible. We will show
that the new $k$-tuple is $p$-admissible for every prime $p$.
There are two cases.

\noindent \textbf{Case 1}: Suppose that $p|r$. Since $\gcd(r_i,r_j)=1$ for $i\neq j$, 
we have that $p|r_\ell$ for exactly one index $\ell$. Now
\[K_\ell(0) = L_{\ell}(x)/r_\ell \equiv 1 \pmod{r_{\ell}} \]
so $K_\ell(0)\equiv 1 \not\equiv 0 \pmod{p}$. We claim that also 
$K_i(0)\not\equiv 0 \pmod{p}$  when $i\neq \ell$. Suppose, by way of 
contradiction, that $K_i(0)\equiv 0 \pmod{p}$ for some $i\neq \ell$. Then $L_i(x) \equiv 0 \pmod{p}$ since $r_i \not\equiv 0 \pmod{p}$, but
$L_{\ell}(x) \equiv r_{\ell}\equiv 0 \pmod{p}$, so 
\[ \det(L_{\ell},L_i)=|a_ib_{\ell} - a_{\ell}b_i| = |a_i L_{\ell}(x) - a_{\ell}L_{i}(x)| \equiv 0 \pmod{p},\]
but this contradicts the assumption that $\gcd(r_{\ell}, \det(L_{\ell},L_i))=1$.
Thus $K_1(0)\cdots K_k(0) \not\equiv 0\pmod{p}$. \\

\noindent \textbf{Case 2}: Now suppose $p\nmid r$. Since $L_1, \ldots, L_k$
is admissible, there is an integer $t_p$ such that $L_1(t_p)\cdots L_k(t_p) \not\equiv 0 \pmod{p}$. Choose $\tau_p$ such that $r\tau_p+x\equiv t_p \pmod{p}$. Then $L_i(r\tau_p+x) \equiv L_i(t_p) \not\equiv 0\pmod{p}$ and $r_i \not\equiv 0\pmod{p}$ for all $i$, so
\[K_1(\tau_p)\cdots K_k(\tau_p) = \frac{L_1(r\tau_p+x)}{r_1}\cdots \frac{L_k(r\tau_p+x)}{r_k} \not\equiv 0\pmod{p}.\]
\end{proof}

Let $n$ be a positive integer and write its prime factorization 
as $n=p_1^{k_1}p_2^{k_2}\cdots p_j^{k_j}$ where the $p_i$ are 
distinct primes with $k_i>0$. Then the \emph{exponent pattern} 
of $n$ is the multiset $\{k_1, k_2, \ldots, k_j\}$ where order 
does not matter but repetitions are allowed. The values of many 
important arithmetic functions depend only on the exponent pattern of the 
input; such functions include:
\[d(x) = \mbox{\# of divisors of $x$}\]
\[\Omega(x) = \mbox{\# of prime factors (counted with multiplicity) of $x$}\]
\[\omega(x) = \mbox{\# of distinct prime factors of $x$}\]
\[\mu(x) = \mbox{M\"{o}bius function} = (-1)^{\omega(x)} \mbox{ if $n$ is squarefree, zero otherwise}\]
\[\lambda(x) = \mbox{Liouville function}= (-1)^{\Omega(x)}\]
Thus if both $x$ and $x+n$ have the same exponent pattern, then
$d(x) = d(x+n)$, $\Omega(x)=\Omega(x+n)$, $\omega(x)=\omega(x+n)$, etc.
In establishing the strong form of the Erd\H{o}s-Mirsky Conjecture (Equation \ref{24}), the authors in \cite{Gold2} actually proved the following result.
\begin{thm}
There are infinitely many positive integers $x$ such that both $x$ and $x+1$ have exponent pattern $\{2,1,1,1\}$.
\end{thm}
We will show that for any shift $n$, there are infinitely many positive integers $x$ such that both $x$ and $x+n$ have a fixed small exponent pattern.
A key tool for doing this is contained in the next remark.

\begin{rem}\label{rem}
Suppose we have an admissible triple of forms $L_i$ with relations $|c_{i,j}L_i-c_{j,i}L_j| = n$. For a given form $L_i$ in the triple, we call $c_{i,j}$ and $c_{i,k}$ where $\{i,j,k\} = \{1,2,3\}$
the \emph{pair of relation coefficients for $L_i$ in the triple}.
Suppose these pairs of relation coefficients for
each form in the triple have matching exponent patterns, i.e.,
$c_{i,j}$ and $c_{i,k}$ have the same exponent pattern with any choices of $i,j,k$ such that
$\{i,j,k\} = \{1,2,3\}$. We then can choose pairwise coprime integers having any desired exponent pattern
which are relatively prime to all linear coefficients and determinants (since determinants of distinct reduced forms are always nonzero). In particular,
we can adjoin integers to the relation coefficients so that the new triple has the property that all of its
relation coefficients have any given exponent pattern $\mathscr{P}$ which contains the exponent patterns of every $c_{i,j}$.
Hence by Theorem \ref{basic}, we would then get infinitely many positive integers $x$ such that both $x$ and $x+n$ have exponent pattern
$\mathscr{P}\cup \{1,1\}$. The proofs of Theorems \ref{shift1} and \ref{shift2} below
will rely heavily on this idea.
\end{rem}

\section{Shifts which are Even or Not Divisible by 15}

\begin{thm}\label{shift1}
Let $n$ be a positive integer with $2|n$ or $15\nmid n$.
Then there are infinitely many positive integers $x$ such that both $x$ and 
$x+n$ have exponent pattern $\{2,1,1,1,1\}$.
\end{thm}
\begin{proof}
Consider the following triple of linear forms: $L_1 =2m+n, L_2=3m+n,$ and $L_3=5m+2n$. We have the relations
\begin{align*}
3L_1 -2L_2&=n \\
5L_1-2L_3&=n \\
3L_3-5L_2&=n
\end{align*}
\noindent Now define $g_i = \gcd(i,n)$ and reduce the linear forms: take $\widetilde{L}_1 = L_1/g_2$, $\widetilde{L}_2 = L_2/g_3$, and $\widetilde{L}_3 = L_3/g_5$.
Then the relations become
\begin{align*}
3\cdot g_2 \widetilde{L}_1 - 2\cdot g_3\widetilde{L}_2&=n \\
5\cdot g_2\widetilde{L}_1 - 2\cdot g_5\widetilde{L}_3&=n \\
3\cdot g_5\widetilde{L}_3 - 5\cdot g_3\widetilde{L}_2&=n   
\end{align*}
\noindent \textbf{Case 1}: Suppose $n$ is even and write $n=2n_2$. Then $g_2=2$, so $\widetilde{L}_1 = m+n_2$, $\widetilde{L}_2 = (3/g_3)m+2(n_2/g_3)$, and $\widetilde{L}_3 = (5/g_5)m+4(n_2/g_5)$. \\

\noindent \textbf{Subcase 1a}: Suppose $2 \mid n_2$. Then 
$$\widetilde{L}_1(1)\widetilde{L}_2(1)\widetilde{L}_3(1) \equiv 1^{3} \not\equiv 0 \pmod{2},$$ so the triple $\widetilde{L}_1$, $\widetilde{L}_2$, $\widetilde{L}_3$ is 2-admissible. Now we check this triple is also $3$-admissible (and therefore admissible).

\hspace{1 in}

\noindent $\bullet$ If $3 \nmid n_2$, then
    $$\widetilde{L}_1(0)\widetilde{L}_2(0)\widetilde{L}_3(0) \equiv n_2(-n_2)(n_2/g_5) \not\equiv 0 \pmod{3}.$$
\noindent $\bullet$ If $3 \mid n_2$, then $g_3=3$, so
$\widetilde{L}_1 \equiv m \equiv \pm \widetilde{L}_3 \pmod{3}$. Now choose $m_0 \in \{1,-1\}$ such that $\widetilde{L}_2(m_0) \not\equiv 0 \pmod{3}$. Then
     $$\widetilde{L}_1(m_0)\widetilde{L}_2(m_0)\widetilde{L}_3(m_0) \equiv m_0 \cdot \widetilde{L}_2(m_0)\cdot (\pm m_0)  \not\equiv 0 \pmod{3}.$$
Here the relation coefficients match in pairs for a given form
in the triple and all have exponent patterns contained 
in $\{1,1\}$, so by appeal to Remark \ref{rem} we have a slightly stronger result, namely, there are infinitely many positive 
integers $x$ such that both $x$ and $x+n$ have exponent pattern 
$\{1,1,1,1\}$. \\

\noindent \textbf{Subcase 1b}: Suppose now $2\nmid n_2$.
Let
\begin{align*}
    K_1 &= \widetilde{L}_1(4m+n_2)/2 = 2m + n_2\\
    K_2 &= \widetilde{L}_2(4m+n_2) = 4\cdot \frac{3}{g_3}m + 5\cdot \frac{n_2}{g_3} \\
    K_3 &= \widetilde{L}_3(4m+n_2) = 4\cdot \frac{5}{g_5}m + 9\cdot \frac{n_2}{g_5}
\end{align*}
Our relations thus become
\begin{align*}
    2^{2}\cdot 3K_1 - 2\cdot g_3K_2 &= n \\
    2^{2}\cdot 5K_1 - 2\cdot g_5K_3 &= n \\
    3\cdot g_5K_3 - 5\cdot g_3K_2 &= n
\end{align*}
Here the pairs of relation coefficients for each form in the triple have matching exponent patterns.
We will check that the triple $K_1, K_2, K_3$ is admissible. First, we note that each form is still reduced: 
$$ K_1 = 2m + n_2$$
is reduced since $2 \nmid n_2$.

$$K_2 = 4\cdot \frac{3}{g_3}m + 5\cdot \frac{n_2}{g_3}$$
is reduced since the constant term is odd and not divisible by $3$ if $g_3 = 1$.
\\
$$K_3 = 4\cdot \frac{5}{g_5}m + 9\cdot \frac{n_2}{g_5}$$
is reduced since the constant term is odd and not divisible by $5$ if $g_5 = 1.$\\

\noindent Next $K_1K_2K_3 \equiv 1\pmod{2}$, so the triple is indeed 2-admissible. Now we check that this triple is 3-admissible.

\hspace{1 in}

\noindent $\bullet$ If $3 \nmid n_2$, then $g_3=1$, so
$$K_1(-n_2)K_2(-n_2)K_3(-n_2)\equiv (-n_2)^2(n_2/g_5) \not\equiv 0 \pmod{3}$$
\noindent $\bullet$ If $3 \mid n_2$, then $K_1 K_3 \equiv \pm m^2 \pmod{3}$. Choose $m_0 \in \{1, -1\}$ such that
$K_2(m_0) \not\equiv 0 \pmod{3}$. Then
$$K_1(m_0)K_2(m_0)K_3(m_0)\equiv \pm (m_0)^2K_2(m_0) \not\equiv 0 \pmod{3}$$.

\noindent Here the relation coefficients all have exponent patterns contained
in $\{2,1,1\}$, so adjoining primes again gives us the statement of the theorem. \\

\noindent \textbf{Case 2:} Now suppose $n$ is odd, so $g_2=1$ from now on. 
Our relations for $\widetilde{L}_i$ become
\begin{align*}
3\widetilde{L}_1 - 2\cdot g_3\widetilde{L}_2&=n \\
5\widetilde{L}_1 - 2\cdot g_5\widetilde{L}_3&=n \\
3\cdot g_5\widetilde{L}_3 - 5 \cdot g_3 \widetilde{L}_2&=n  
\end{align*}
If we look at this modulo $2$, we get $\widetilde{L}_1 \equiv 1, \widetilde{L}_2 \equiv m+1, \widetilde{L}_3 \equiv m.$
Thus this triple is not 2-admissible here. However, we can restrict $m \pmod{2}$ and reduce to get $2$-admissible. To do this, we write
\begin{align*}
    M_1 &= \widetilde{L}_1(2m) = 4m + n \\
    M_2 &= \widetilde{L}_2(2m) = 2 \cdot \frac{3}{g_3}m + \frac{n}{g_3} \\
    M_3 &= \widetilde{L}_3(2m)/2 = \frac{5}{g_5}m + \frac{n}{g_5}.
\end{align*}
The triple $M_1, M_2, M_3$ has reduced forms and is $2$-admissible with relations
\begin{align*}
    3M_1 - 2\cdot g_3M_2 &= n \\
    5M_1 - 2^2 \cdot g_5 M_3 &= n \\
    2\cdot 3\cdot g_5 M_3 - 5\cdot g_3 M_2 &= n
\end{align*}
Note, however, that the relation coefficients for $M_3$ do not have matching
exponent patterns. We can remedy this by restricting and reducing modulo $3$. \\

\noindent \textbf{Subcase 2a}: Suppose $3 \nmid n$, so $g_3=1$. Take
\begin{align*}
    N_1 &= M_1(3m+n) = 12m+5n \\
    N_2 &= M_2(3m+n) = 18m +7n\\
    N_3 &= M_3(3m+n)/3 = \frac{5}{g_5}m+2\cdot \frac{n}{g_5}
\end{align*}
Now we get relations
\begin{align*}
    3N_1 - 2 N_2 &= n \\
    5N_1 - 2^2 \cdot 3\cdot g_5 N_3 &= n \\
    2\cdot 3^2\cdot g_5 N_3 - 5 N_2 &= n
\end{align*}
All these forms are reduced and the triple is still $2$-admissible since $N_1(1)N_2(1)N_3(1)\equiv 1^3\not\equiv 0 \pmod{2}$. In fact, the triple
is $3$-admissible too since
$$N_1(0)N_2(0)N_3(0) \equiv (-n)(n)(-n/g_5) \not\equiv 0 \pmod{3}.$$
Here the relation coefficients all have exponent 
patterns contained
in $\{2,1,1\}$, so adjoining primes again gives us the statement of the 
theorem. In fact, if we also have $5\nmid n$ here,
then the relation coefficients all have exponent patterns contained
in $\{2,1\}$ so we get infinitely many positive integers $x$ such that $x$ and $x+n$ both have exponent pattern $\{2,1,1,1\}$.\\

\noindent \textbf{Subcase 2b}: Suppose now $3 \mid n$, so $5 \nmid n$ by our assumption
that $15\nmid n$. We still must factor out a $3$ from $M_3$, but doing so
will force us to also factor out a $3$ from $M_1$ which then tells us to also 
factor out a $5$ from $M_1$ to make its pair of relation coefficients in the triple have
matching exponent patterns. Thus we will restrict modulo $15$: write $n=3n_3$ and take
\begin{align*}
    J_1 &= M_1(15m-4n)/15 = 4m-n \\
    J_2 &= M_2(15m-4n)/(g_9/3) = 10\cdot \frac{9}{g_9} m -23 \cdot \frac{n}{g_9} \\
    J_3 &= M_3(15m-4n)/3 = 25m-19n_3
\end{align*}
where, as indicated above, $g_9 = \gcd(9,n)$ which is either $3$ or $9$ in this 
case. Here we have relations
\begin{align*}
    3^2\cdot 5 J_1 - 2\cdot g_9 J_2 &= n \\
    3\cdot 5^2 J_1 - 2^2 \cdot 3 J_3 &= n \\
    2\cdot 3^2 J_3 - 5\cdot g_9 J_2 &= n
\end{align*}
All the forms are reduced (since $5\nmid n$) and the triple is $2$-admissible since $J_1(0)J_2(0)J_3(0) \equiv 1^3\not\equiv 0 \pmod{2}$. \\

\noindent Now we check that this triple is 3-admissible.

\hspace{1 in}

\noindent $\bullet$ If $3 \nmid n_3$, then $g_9=3$, so
$$J_1(-n_3)J_2(-n_3)J_3(-n_3)\equiv (-n_3)(n_3)^2 \not\equiv 0 \pmod{3}.$$

\noindent $\bullet$ If $3 \mid n_3$, then $g_9=9$ so
$J_1 J_3 \equiv m^2 \pmod{3}$. Choose $m_0 \in \{1,-1\}$ such that
$J_2(m_0) \not\equiv 0 \pmod{3}$. Then
$$J_1(m_0)J_2(m_0)J_3(m_0)\equiv (m_0)^2J_2(m_0) \not\equiv 0 \pmod{3}.$$

\noindent Here the relation coefficients all have exponent patterns contained
in $\{2,1,1\}$ (or even in $\{2,1\}$ in the case that $9|n$), so adjoining primes again gives us the statement of the theorem.
\end{proof}
\begin{rem}
If we assume the twin prime conjecture, then for any positive integer $n$,
there are primes $p$ and $p+2$ such that neither divide $15n$. In this case,
we can use the following triple: $L_1=2m+n$, $L_2=pm+n(p-1)/2$, $L_3=(p+2)m+n(p+1)/2$. Building off this triple will show---as in Subcase 2a above---that there
are infinitely many positive integers $x$ such that $x$ and $x+n$ both have exponent pattern $\{2,1,1,1\}$. We will not include the details here since we
give an unconditional proof of a result for the remaining case not covered by Theorem \ref{shift1}.
\end{rem}
\section{Shifts which are Odd and Divisible by 15}
\begin{thm}\label{shift2}
Let $n$ be a positive integer with $2\nmid n$ and $15|n$. Then there 
are infinitely many positive integers $x$ such both $x$
and $x+n$ have exponent pattern $\{3,2,1,1,1,1,1\}$.
 \end{thm}
\begin{proof}
By considering the admissible triple $m, m+4, m+10$, we find that for any constant 
$C$ there are infinitely many pairs of $E_2$ numbers each having prime factors bigger than $C$ and which are a distance of either $4$, $6$, or $10$ apart. In particular,
there are odd $E_2$ numbers $q_1$, $q_2$ such that $\gcd(q_i,n)=1$ for $i=1,2$
and $q_2 = q_1+2j$ where $j\in \{2,3,5\}$. Thus we may write $q_1=p_{1,1}p_{1,2}$
and $q_2=p_{2,1}p_{2,2}$ where $p_{1,1}$, $p_{1,2}$, $p_{2,1}$, and $p_{2,2}$ are all
distinct primes, none of which divide $2n$. There are integers $a$, $b$ with $a$ even
and $b$ odd such that $-aq_2+bq_1=1$.
Write 
$a=2a_2$ and define the triple of linear forms
\begin{align*}
    L_1 &= q_1m+a_2n \\
    L_2 &= 2q_2m+bn \\
    L_3 &= 4\cdot \frac{j}{g} m +(b-a)\frac{n}{g}
\end{align*}
where $g=1$ if $j=2$ and $g=j$ otherwise. Now we check 
that this triple is admissible. We only need to check for $2$-admissible 
and $3$-admissible since each form is reduced by construction. The triple is $2$-admissible since $L_1\cdot L_2 \cdot L_3 \equiv L_1 \cdot 1 \cdot 1 \pmod{2}$. To check the triple is $3$-admissible, choose $m_0 \in \{1,-1\}$ with
$L_3(m_0) \not\equiv 0 \pmod{3}$. Then $L_1(m_0) L_2(m_0) L_3(m_0) \equiv (q_1m_0)(-q_2m_0)L_3(m_0) \not\equiv 0 \pmod{3}$. Moreover, the triple satisfies the
relations
\begin{align}
    q_1L_2-2q_2L_1 &= n \\
    gq_1L_3-2^2jL_1 &= n \\
    gq_2L_3-2jL_2 &= n
\end{align}
However, the pairs of relation coefficients for $L_1$, $L_2$ do not have matching
exponent patterns in the triple, so we will need to adjoin primes using Theorem \ref{adjoin}.
We will break up the proof into cases depending on the value of $j$, but in both cases
we need to note that the pairwise determinants are relatively prime to the integers
we want to adjoin:
\begin{align*}
\det(L_1,L_2) &= q_1bn-2a_2nq_2 = n \\
\det(L_1,L_3) &= q_1(b-a)\frac{n}{g}-4a_2n\cdot \frac{j}{g} = \frac{n}{g}\\
\det(L_2,L_3) &= 2q_2(b-a)\frac{n}{g}-4bn\cdot \frac{j}{g} = 2\cdot \frac{n}{g}
\end{align*}

\noindent \textbf{Case 1}: Suppose $j=2$, so $g=1$.

We apply Theorem \ref{adjoin} directly with $r_1=p_{2,1}^2p_{2,2}$, $r_2=p_{1,1}$, and $r_3=1$, so
we get a new admissible triple of forms $K_i$ which satisfies the following relations:

\begin{align*}
    |p_{1,1}^2p_{1,2}K_2-2p_{2,1}^3p_{2,2}^2K_1| &= n \\
    |q_1K_3-2^3p_{2,1}^2p_{2,2}K_1| &= n \\
    |q_2K_3-2^2p_{1,1}K_2| &= n.
\end{align*}
Here the relation coefficients of $K_1$ both have exponent pattern $\{3,2,1\}$, the relation coefficients of $K_2$ both have exponent pattern $\{2,1\}$, and the relation coefficients of $K_3$ both have exponent pattern $\{1,1\}$. Thus by another application of Theorem \ref{adjoin} via Remark \ref{rem} we can arrange
an admissible triple with common relation value $n$ and all relation 
coefficients having exponent pattern $\{3,2,1,1,1\}$ (or even $\{3,2,1,1\}$
in this case). \\ 

\noindent \textbf{Case 2}: Suppose $j\neq 2$, so $g=j$.
We apply Theorem \ref{adjoin} directly with $r_1=p_{2,1}$, and $r_2=r_3=1$, so
we get a new admissible triple of forms $K_i$ which satisfies the following relations:

\begin{align*}
    |q_1K_2-2p_{2,1}^2p_{2,2}K_1| &= n \\
    |jq_1K_3-2^2jp_{2,1}K_1| &= n \\
    |jq_2K_3-2jK_2| &= n.
\end{align*}
Here the relation coefficients of $K_1$ both have exponent pattern $\{2,1,1\}$, the relation coefficients of $K_2$ both have exponent pattern $\{1,1\}$, and the relation coefficients of $K_3$ both have exponent pattern $\{1,1,1\}$. Thus by appeal to Theorem \ref{adjoin} via Remark \ref{rem} we can arrange
an admissible triple with common relation value $n$ and all relation 
coefficients having exponent pattern $\{3,2,1,1,1\}$ (or even $\{2,1,1,1\}$
in this case). \\

Therefore, in either case, there are infinitely many pairs of positive 
integers both having exponent pattern $\{3,2,1,1,1,1,1\}$ which are a 
distance of $n$ apart.
\end{proof}

\bibliographystyle{amsalpha}
\providecommand{\bysame}{\leavevmode\hbox to3em{\hrulefill}\thinspace}
\providecommand{\MR}{\relax\ifhmode\unskip\space\fi MR }
\providecommand{\MRhref}[2]{%
  \href{http://www.ams.org/mathscinet-getitem?mr=#1}{#2}
}
\providecommand{\href}[2]{#2}

\end{document}